\UseRawInputEncoding
\documentclass[11pt,a4paper]{article}
\usepackage{amsfonts}
\usepackage{amssymb}
\usepackage{mathrsfs}
\usepackage{amsmath}
\usepackage{booktabs}
\usepackage{epsf,epsfig,amsfonts,amsgen,indentfirst}
\usepackage{amsmath,amstext,amsbsy,amsopn,amsthm,bbding,wasysym}
\usepackage{multicol,mathdots}
\usepackage{subfigure}
\allowdisplaybreaks

\newtheorem{theorem}{Theorem}[section]
\newtheorem{corollary}{Corollary}[section]

\newtheorem{lemma}[theorem]{Lemma}

\newtheorem{claim}{Claim}

\newtheorem{conjecture}[theorem]{Conjecture}
\newtheorem{problem}[theorem]{Problem}

\newtheorem{remark}[claim]{Remark}

\begin{document}
\title
{\LARGE \textbf{Some results on total weight choosability\thanks{ supported by NSFC (No. 12261071), NSF of Qinghai Province (No. 2020-ZJ-920), NSFC Grant (No. 11901263, 12071194, 12271228) and NSFC of Gansu Province (No. 21JR7RA511).} }}

\author{ Tingzeng Wu$^{a,b}$\thanks{{Corresponding author.\newline
\emph{E-mail address}: mathtzwu@163.com, jianxuanluo@163.com, gaoyp@lzu.edu.cn
}},  Jianxuan Luo$^{a}$, Yuping Gao$^{c}$\\
{\small $^{a}$ School of Mathematics and Statistics, Qinghai Nationalities University, }\\
{\small  Xining, Qinghai 810007, P.R.~China} \\
{\small $^{b}$Qinghai Institute of Applied Mathematics, Xining, Qinghai, 810007, P.R.~China} \\
{\small $^{c}$School of Mathematics and Statistics, Lanzhou University,} \\
{\small  Lanzhou, Gansu, 730000, P.R.~China}}

\date{}

\maketitle
\noindent {\bf Abstract:} A graph $G=(V,E)$ is called
$(k,k')$-choosable if
for any total list assignment $L$
which assigns to each vertex $v$ a set $L(v)$ of $k$ real numbers,
and assigns to each edge $e$ a set $L(e)$ of $k'$ real numbers,
there is a mapping $f:V\cup E\rightarrow \mathbb{R}$
such that $f(y)\in L(y)$ for any $y\in V\cup E$
and for any two adjacent vertices $v, v'$,
$\sum_{e\in E(v)}f(e)+f(v)\neq \sum_{e\in E(v')}f(e)+f(v')$, where $E(x)$ denotes the set of incident edges of a vertex $x\in V(G)$.
In this paper, we characterize  a sufficient condition on
 $(1,2)$-choosable of graphs.
We show that every connected $(n,m)$-graph is both $(2,2)$-choosable and $(1,3)$-choosable if  $m=n$ or $n+1$,
where $(n,m)$-graph denotes the graph with $n$ vertices and $m$ edges.
Furthermore, we prove that some  graphs obtained by
some graph operations are $(2,2)$-choosable.

\noindent {\bf Keywords:} Total-weighting; List assignment; Permanent index; Line graph; Graph operation
\section{Introduction}%
Let $G=(V(G),E(G))$ be a graph with $n$ vertices. The number of edges in $G$ is denoted by $m(G)$ or $m$ for short. We also call $G$ as an $(n,m)$-\emph{graph}.
For a subgraph $H$ of $G$,
let $G-E(H)$ denotes the subgraph obtained from $G$ by deleting the edges of $H$.
A \emph{matching} in a graph is a set of non-loop edges with no common endvertices, and an endvertex in an edge of a matching is said to be \emph{saturated} by the matching. A \emph{perfect matching} in a graph is a matching that saturates every vertex.
The number of perfect matchings of $G$ is denoted by $M(G)$.
For convenience, a path, a cycle and a complete graph with $n$ vertices are denoted by $P_{n}$, $C_{n}$ and $K_{n}$, respectively.

A \emph{total-weighting} of a graph $G$ is a mapping $f$: $V\cup E \rightarrow \mathbb{R}$
which assigns to each vertex and each edge a real number as its weight.
For a total-weighting $f$,
we use $s(v)=f(v)+\sum_{e\in E(v)}f(e)$ to denote the weight of a vertex $v\in V(G)$, where $E(v)$ denotes the set of edges incident with $v$.
If $s(u)\neq s(v)$ for any two adjacent vertices $u,v\in V(G)$, then we call the total-weighting $f$ \emph{proper}.

The list version of total-weighting of graphs was introduced independently by Przyby\l o and Wo\'{z}niak \cite{prz2} and Wong and Zhu \cite{wong2}.
Let $\psi: V\cup E\rightarrow \mathbb{N}^{+} $. A $\psi$-\emph{list assignment} of a graph $G$ is a mapping $L$ which assigns to $z\in V\cup E$ a set $L(z)$ of $\psi(z)$ real numbers. Given a total list assignment $L$, a proper $L$-total weighting is a \emph{proper total
weighting} $\varphi$ with $\varphi(z)\in L(z)$ for all $z\in V\cup E$. We say $G$ is \emph{total weight} $\psi$-\emph{choosable} ($\psi$-\emph{choosable} for short) if for any $\psi$-list assignment $L$, there is a proper $L$-total weighting
of $G$. We say $G$ is \emph{total weight} $(k,k')$-\emph{choosable} ($(k,k')$-choosable for short) if $G$ is $\psi$-total weight choosable, where $\psi(v)=k$ for $v\in V (G)$ and $\psi(e) = k'$ for $e\in E(G)$.

Wong and Zhu \cite{wong2} proposed two Conjectures as follows:

\begin{conjecture}\label{art11}\cite{wong2}
Every graph with no isolated edges is $(1,3)$-choosable.
\end{conjecture}

\begin{conjecture}\label{art12}\cite{wong2}
Every graph is $(2,2)$-choosable.
\end{conjecture}

The \emph{permanent} of an $m\times m$ real matrix $A=[a_{ij}]$,
with $i,j \in \{1,2,\ldots,m\}$, is defined as
$$
{\rm per}(A)=\sum_{\sigma}\prod_{i=1}^{m}a_{i\sigma(i)},
$$
where the summation takes over all permutations $\sigma$ of
$\{1,2,\ldots, m\}$.

The {\em permanent index} of a matrix $A$, denoted by ${\rm pind}(A)$,
is the minimum integer $k$ such that
there exists a matrix $A'$ with ${\rm per}(A')\neq 0$,
each column of $A'$ is a column of $A$
and each column of $A$ occurs in $A'$ at most $k$ times.
Let $A_{G}$ be a matrix
with rows indexed by the edges of $G$
and columns indexed by the vertices and edges of $G$,
where if $e=(u,v)$(oriented from $u$ to $v$), then
\begin{eqnarray*}
A_G[e,y]=
\begin{cases}
1& \text{if $y=v$, or $y\neq e$ is an edge incident to $v$},\\
-1& \text{if $y=u$, or $y\neq e$ is an edge incident to $u$},\\
0& \text{otherwise}.\\
\end{cases}
\end{eqnarray*}
and let $B_{G}$ be the submatrix of $A_{G}$
with those columns of $A_{G}$ indexed by edges.

An {\em index function} of $G$ is a mapping $\eta$,
and it assigns to every vertex or edge $z$ of $G$
a non-negative integer $\eta$.
If $\sum_{y\in V(G)\cup E(G)}\eta(z)=|E(G)|$,
then the index function $\eta$ is {\em valid}.
For an index function $\eta$ of $G$, denote by $A_{(\eta)}$ the matrix,
each of its column is a column of $A_G$,
and each column $A_G(z)$ of $A_G$ can appear up to $\eta(z)$ times in $A_{(\eta)}$.
It is shown in \cite{alon2}  and \cite{wong2} that $G$ is $(2,2)$-choosable if ${\rm pind}(A_{G})=1$,
and $G$ is $(1,3)$-choosable if ${\rm pind}(B_{G})\leq 2$.

Bartnicki et al. \cite{bart} and Wong et al. \cite{wong2} proposed two Conjectures independently
as follows:

\begin{conjecture}\label{art13}\cite{bart}
For any graph $G$ with no isolated edges, ${\rm pind}(B_{G})\leq 2$.
\end{conjecture}

\begin{conjecture}\label{art14}\cite{wong2}
For any graph $G$, ${\rm pind}(A_{G})=1$.
\end{conjecture}

The above two conjectures have received a lot of attention.
However, they have not been solved yet,
which can only be proved to be true for some special graphs.
Recently, it was proved in \cite{zhu1}
that every graph with no isolated edges is $(1,5)$-choosable.
Some special graphs are shown to be $(2,2)$-choosable,
such as trees, complete graphs \cite{wong2},
subcubic graphs, 2-trees, Halin graphs, grids \cite{wong1}.
Some special graphs are shown to be $(1,3)$-choosable,
such as complete graphs, complete bipartite graphs,
trees without $K_{2}$ \cite{bart},
Cartesian product of an even number of even cycles,
of $P_{n}$ and an even cycle, of two paths \cite{wong3}.

Wong and Zhu \cite{wong2}
showed that if a graph is $(k,k')$-choosable
then it is $(k+1,k')$-choosable and $(k,k'+1)$-choosable.
Hence there is a natural problem as follows.

\begin{problem}\label{art15}
Characterizing  graphs that are $(1,2)$-choosable.
\end{problem}
In response to the above problem,
some results have been obtained.
Wong et al. \cite{wong5} proved that complete bipartite graphs without $K_{2}$
are $(1,2)$-choosable;
Chang et al. \cite{chang} proved that a tree with even number of edges is $(1,2)$-choosable.

In this paper, we focus on Problem \ref{art15} and Conjectures \ref{art13} and \ref{art14}, we show that some graphs are $(2,2)$-choosable
as well as $(1,3)$-choosable.
The remainder of this paper is organized as follows.
In Section~2, we determine a sufficient condition for a graph to be  $(1,2)$-choosable. As  applications, we show that an $(n,m)$-graph is  $(1,2)$-choosable when $m=n-1$, $n$ and $n+1$.
In Section~3, we prove that all $(n,m)$-graphs are $(2,2)$-choosable
as well as $(1,3)$-choosable, where $m=n$ and $n+1$.
In the final section, we prove that some graphs under
some graph operations are $(2,2)$-choosable.

\section{A solution to Problem~\ref{art15}}

In this section, we will characterize a sufficient condition to answer Problem \ref{art15}.
Chang et al.~\cite{chang} gave an important result on $(1,2)$-choosable  of graphs as follows.
\begin{lemma}(\cite{chang})\label{art21}
If ${\rm per}(B_{G})\neq 0$.
Then $G$ is $(1,2)$-choosable.
\end{lemma}

A Sachs graph is a simple graph such that each component is regular and has degree 1 or 2.
In other words the components are single edges and cycles.
Merris et al. \cite{mer}
gave a formula for calculating the permanent of any graph $G$:
\begin{eqnarray*}
{\rm per}(A(G))=|(-1)^n\sum_H2^{k(H)}|,
\end{eqnarray*}
where the summation takes over all Sachs subgraphs $H$ of order $n$ in $G$,
and $k(H)$ is the number of cycles in $H$.

\begin{theorem}\label{art22}
Let $G$ be a connected graph with $m$ edges. If the number of perfect matchings in the line graph $L(G)$ of $G$ is odd,
then $G$ is $(1,2)$-choosable.
\end{theorem}
\begin{proof}
Replace $-1$ by $1$ in $B_{G}$ and the obtained matrix is just the adjacent matrix $A(L(G))$ of $L(G)$. It can be seen that ${\rm per}(B_{G})\equiv {\rm per}(A(L(G)))({\rm mod} 2)$. According to formula $(1)$, we get that
\begin{eqnarray*}
{\rm per}(A(L(G)))=|(-1)^m\sum_H2^{k(H)}|=M(L(G))+\sum_{H'}2^{k(H')},
\end{eqnarray*}
where $H'$ denotes the Sachs subgraphs of $m$ vertices containing cycles of line graph $L(G)$.
Thus,
\begin{eqnarray*}
{\rm per}(A(L(G)))\equiv M(L(G))({\rm mod}~2).
\end{eqnarray*}
Furthermore,
\begin{eqnarray*}
{\rm per}(B_{G})\equiv M(L(G))({\rm mod}~2).
\end{eqnarray*}
By Lemma \ref{art21} and the above equation,  $G$ is $(1,2)$-choosable if $ M(L(G))$ is odd.
\end{proof}

As applications of Theorem \ref{art22}, we will show that some $(n,m)$-graphs are $(1,2)$-choosable when $m=n-1$, $n$ and $n+1$ as follows.

Obviously, a connected $(n,m)$-graph is a tree when $m=n-1$.
Chang et al. \cite{chang} proved that a tree with even number of edges is $(1,2)$-choosable.
According to Theorem \ref{art22},
we can give a new proof. To achieve it,
we first introduce some lemmas as follows.

For any graph $G$,
let $p(G)$ be the number of components of $G$
which have an even number of edges.
If $G$ is a forest, $p(G)$ and $|V(G)|$ have the same parity.
Thus, if $G$ is a tree and $|V(G)|$ is odd,
then $p(G-v)$ is even for all $v\in V(G)$.
For any non-negative $k$, denote by $(2k)!!=\frac{(2k)!}{k!\times 2^k}$.

\begin{lemma}(\cite{dong})\label{art23}
Let $T$ be a tree with
$V(T)=\{v_1,v_2,\ldots,v_n\}$,
where $n>1$ is odd. Then
\begin{equation*}
M(L(T))=\prod_{i=1}^np(T-v_i)!!.
\end{equation*}
\end{lemma}

\begin{lemma}\label{art24}
 $\frac{(2k)!}{k!\times 2^k}=(2k-1)\times(2k-3)\times\ldots\times3\times1$, where $k$ is a non-negative integer.
\end{lemma}

By Theorem \ref{art22}, Lemmas \ref{art23} and \ref{art24},
we can get a result as follows.

\begin{theorem}(\cite{chang})\label{art25}
If $T$ is a tree with even number of edges.
Then $T$ is $(1,2)$-choosable.
\end{theorem}

Next, we give a recursive expression for $M(L(G))$.
Let $e$ be any edge of $G$ with endvertices $u$ and $v$.
Let $G(u,w)$ be the graph obtained from $G-e$
by adding a new vertex $w$ and adding a new edge joining $w$ to $u$.
$G(v,w)$ is defined similarly.

\begin{lemma}(\cite{dong})\label{art26}
Let $G$ be a graph, and let $e=uv$ be an edge of $G$.
Then
\begin{eqnarray*}
M(L(G))=M(L(G(u,w)))+M(L(G(v,w))).
\end{eqnarray*}
\end{lemma}

A vertex of degree one is called a \emph{leaf} in a graph. A \emph{unicyclic graph} is a connected graph containing exactly one cycle, the cycle denoted by $C_l$.
Obviously, a connected $(n,m)$-graph is unicyclic if and only if $n=m$. The set of unicyclic graphs with $n$ vertices is denoted by $\mathscr{U}_{n}$.
For any graph $U\in \mathscr{U}_{n}$ with $V(C_l)=\{v_1,\ldots,v_l\}\subseteq V(U)$, $U$ can be viewed as identifying $v_i$ with any leaf of each of $k_i$ trees for $i\in\{1,\ldots,l\}$, where $k_i$ is a non-negative integer.
Denote by $k_i^{0}(\geq2)$ and $k_{i}^{1}(\geq3)$ respectively the number of trees with even number of edges and odd number of edges in the $k_i$ trees. Let $s=\sum_{i=1}^l k_i^0$, $\mathscr{U}_{1}$ be the subset of $\mathscr{U}_{n}$ such that $s$ is odd and $\mathscr{U}_{2}$ be the subset of $\mathscr{U}_{n}$ such that $s$ is even.
We denote the $s$ trees with even number of edges as $T_1,T_2,\ldots,T_s$,
respectively. As we will consider the number of perfect matchings of line graphs, assume that $n$ is even and all the notation in this paragraph is followed in Theorem \ref{art28} and Lemma \ref{art27}.

\begin{theorem}\label{art28}
For any graph $U\in \mathscr{U}_{1}$, $U$ is $(1,2)$-choosable.
\end{theorem}
\begin{proof} By Lemma \ref{art26},
\begin{eqnarray*}
M(L(U))=M(L(T_1))+M(L(T_2))+\ldots+M(L(T_s))+M(L(U')).
\end{eqnarray*}
where $e_i$ denotes the edge incident with $T_i$ and $v_i$, $U'=U-E(T_1-e_1)-E(T_2-e_2)-\ldots-E(T_s-e_s)$.

According to the definition of $\mathscr{U}_{1}$,
 $s$ and $m(T_i-e_i)$ are odd, $m(U)$ is even.
So
$m(U')=m(U)-m(T_1-e_1)-m(T_2-e_2)-\ldots-m(T_s-e_s)$ is odd,
then $M(L(U'))=0$.
By Theorem \ref{art25}, $M(L(T_i))$ is odd as $m(T_i)$ is even.
From the above argument and equation,
we obtain that $M(L(U))$ is odd. Then $U\in \mathscr{U}_{1}$ is $(1,2)$-choosable
by Theorem \ref{art22}.
\end{proof}

\begin{lemma}\label{art27}
Let $U\in \mathscr{U}_{2}$. Then $M(L(U))$ is even.
\end{lemma}
\begin{proof}
By Lemma \ref{art26},
\begin{eqnarray*}
M(L(U))&=&M(L(T_1))+\ldots+M(L(T_s))+M(L(U'))\\
&=&M(L(T_1))+\ldots+M(L(T_s))+M(L(U'(u,w)))+M(L(U'(v,w))),
\end{eqnarray*}
where $e_i$ denotes the edge incident with $T_i$ and $v_i$,
$U'=U-E(T_1-e_1)-E(T_2-e_2)-\ldots-E(T_s-e_s)$, $e=uv$ denotes any edge of $C_l$ in $U'$ and $U'(x,w)$ is the graph obtained from $G-e$
by adding a new vertex $w$ and adding a new edge $wx$ for $x\in\{u,v\}$.

By the above definition of $\mathscr{U}_{2}$,
$m(T_i-e_i)$ is odd, $s$ and $m(U)$ are even.
So, $m(U')=m(U)-m(T_1-e_1)-m(T_2-e_2)-\ldots-m(T_s-e_s)$ is even.
Hence $m(U'(u,w))$ and $m(U'(v,w))$ are even, and $U'(u,w),U'(v,w)$ are trees.
By Theorem \ref{art25}, $M(L(U'(u,w)))$ and $M(L(U'(v,w)))$ are odd.
Since $m(T_i)$ is even, by Theorem \ref{art25}, we have that $M(L(T_i))$ is odd.
From the above argument and equation,
we obtain that $M(L(U))$ is even.
\end{proof}

A connected $(n,m)$-graph containing two or three cycles is called a \emph{bicyclic graph} if $m=n+1$. Let $\mathscr{B}_{n}$ be the set of all bicyclic graphs with $n$ vertices.
By the structure of bicyclic graphs, it is known that $\mathscr{B}_{n}$
consists of three types of graphs:
the first type, denoted by $\mathscr{B}_n^1(p,q)$,
is the set of graphs each of which contains $B_1(p,q)$
as a vertex-induced subgraph;
the second type, denoted by $\mathscr{B}_n^2(p,q,r)$,
is the set of graphs each of which contains $B_2(p,q,r)$
as a vertex-induced subgraph;
the third type, denoted by $\mathscr{B}_n^3(p,q,r)$,
is the set of graphs each of which contains $B_3(p,q,r)$
as a vertex-induced subgraph (see Figure \ref{fig1}).
Obviously, $\mathscr{B}_{n}=\mathscr{B}_n^1(p,q)
\cup\mathscr{B}_n^2(p,q,r)\cup\mathscr{B}_n^3(p,q,r)$.
\begin{figure}[htbp]
\begin{center}
\includegraphics[scale=1.2]{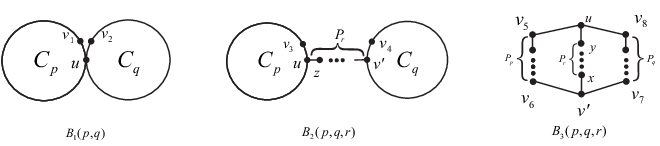}
\caption{\label{fig1}\small
{Bicyclic graphs $B_1(p,q)$, $B_2(p,q,r)$ and $B_3(p,q,r)$}}
\end{center}
\end{figure}

Let $C_p$ and $C_q$ denote the induced cycles of any bicyclic graph in $\mathscr{B}_n^1(p,q)$. For any graph $B\in \mathscr{B}_n^1(p,q)$ with $V(B_1(p,q))=\{v_0, v_1,\ldots, v_{p+q-1}\}\subseteq V(B)$, $B$ can be viewed as identifying $v_i$ with any leaf of each of $k_i$ trees for $i\in\{1,\ldots,p+q-1\}$,
where $k_i$ is a non-negative integer. Denote by $k_i^{0}(\geq2)$ and $k_{i}^{1}(\geq3)$ respectively the number of trees with even number of edges and odd number of edges in the $k_i$ trees. Let $s_1=\sum_{i} k_i^0$, where the summation takes over all vertices of $C_p$ or $C_q$,
$\mathscr{B}_{1}\subset \mathscr{B}_n^1(p,q)$ such that each graph in $\mathscr{B}_{1}$ contains even number of edges and $s_1$ is odd or even. As we will consider the number of perfect matchings of line graphs, assume that $n$ is odd and all the notation in this paragraph and Figure 1 is followed in Lemmas~\ref{art29}, \ref{art210}, \ref{art211}, \ref{art212} and Theorem~\ref{art213}.

\begin{lemma}\label{art29}
Let $B\in \mathscr{B}_{1}$ be a bicyclic graph. Then $B$ is $(1,2)$-choosable.
\end{lemma}
\begin{proof} Let $e=uv_{1}$ as in Figure \ref{fig1}. By Lemma \ref{art26},
\begin{eqnarray*}
M(L(B))=M(L(B(u,w)))+M(L(B(v_{1},w))).
\end{eqnarray*}

Clearly, $B(u,w), B(v_{1},w)\in \mathscr{U}_{n}$.
According to the definition of $\mathscr{B}_{1}$, $m(B)$ is even.
Then the number of edges of
$B(u,w)$ and $B(v_{1},w)$ are even.
Set $s=x$ for $B(u,w)$ and $s=y$ for $B(v_{1},w)$.

Without loss of generality,
assume $s_1$ is odd or even,
where the summation takes over all vertices of $C_q$.
Let $B'$ be obtained from $B$ by deleting the edges in
$C_{q}$ and the trees hanging on it.
Because $m(B')$ is either an even or an odd number,
so we should consider two cases as follows.

\textbf{Case 1.} $m(B')$ is even.

If $s_1$ is odd and $m(B')$ is even,
then $x=s_1$ and $y=s_1+1$.
Then  $M(L(B(u,w)))$ is odd by Theorem \ref{art28}.
According to Lemma \ref{art27}, $M(L(B(v_{1},w)))$ is even.
It follows that $M(L(B))$ is odd by the above equation. Therefore, $B\in \mathscr{B}_{1}$ is $(1,2)$-choosable by Theorem $\ref{art22}$.

If $s_1$ is even and $m(B')$ is even, then $x=s_1$ and $y=s_1+1$.
Then $M(L(B(u,w)))$ is even according to Lemma \ref{art27}. By Theorem \ref{art28}, $M(L(B(v_{1},w)))$ is odd. Therefore,  $M(L(B))$ is odd by the above equation and  $B\in \mathscr{B}_{1}$ is $(1,2)$-choosable by Theorem \ref{art22}.

\textbf{Case 2:} $m(B')$ is odd.

If $s_1$ is odd and $m(B')$ is odd,
then $x=s_1+1$ and $y=s_1$.
Hence, $M(L(B(u,w)))$ is even according to Lemma \ref{art27}.
By Theorem \ref{art28}, $M(L(B(v_{1},w)))$ is odd. Then $M(L(B))$ is odd by the above equation and $B\in \mathscr{B}_{1}$ is $(1,2)$-choosable by Theorem \ref{art22}.

If $s_1$ is even and $m(B')$ is odd,
then $x=s_1+1$ and $y=s_1$.
Hence,  $M(L(B(u,w)))$ is odd by Theorem \ref{art28}.
According to Lemma \ref{art27}, $M(L(B(v_{1},w)))$ is even. Therefore,  $M(L(B))$ is odd by the above equation and $B\in \mathscr{B}_{1}$ is $(1,2)$-choosable by Theorem \ref{art22}.
\end{proof}

Let $C_p$ and $C_q$ denote the induced cycles of any bicyclic graph in $\mathscr{B}_n^2(p,q)$.
Let $\mathscr{B}_{2}\subset \mathscr{B}_n^2(p,q,r)$
be the set of all graphs obtained
by identifying every vertex $v_i$ of $B_2(p,q,2)$ with
any leaf of each of the $k_i$ trees,
where $k_i=k_i^0+k_i^1$
(the number of trees with even number of edges $(\geq 2)$ denoted by $k_i^0$,
the number of trees with odd number of edges $(\geq 3)$ denoted by $k_i^1$)
and such that $\sum_{i} k_i^0=s_1$ is odd,
where the summation takes over all vertices of $C_p$
if $m(B(u,w))$ is even or $C_q$ if $m(B(u,w))$ is odd.
Because we need to consider the number of perfect matchings of line graph,
$n$ is assumed to be odd.

\begin{lemma}\label{art210}
Let $B\in \mathscr{B}_{2}$ be a bicyclic graph. Then $B$ is $(1,2)$-choosable.
\end{lemma}
\begin{proof}
Let $e=uv'$. By Lemma \ref{art26},
\begin{eqnarray*}
M(L(B))=M(L(B(u,w)))+M(L(B(v',w))).
\end{eqnarray*}

Obviously, $B(u,w), B(v',w)\in \mathscr{U}_{n}$.
By the definition of $\mathscr{B}_{2}$, $m(B)$ is even.
We consider two cases as follows.

\textbf{Case 1:} $m(B(u,w))$ is even and $m(B(v',w))$ is odd.

By the above definition, we have $s_1$ is odd
and the summation takes over all vertices of $C_p$. Then $M(L(B(v',w)))=0$ and $M(L(B(u,w)))$ is odd by Theorem \ref{art28}. It follows that $M(L(B))$ is odd by the above equation  and  $B\in \mathscr{B}_{2}$ is $(1,2)$-choosable by Theorem \ref{art22}.

\textbf{Case 2:}  $m(B(u,w))$ is odd and $m(B(v',w))$ is even.

By the above definition, we have $s_1$ is odd
and the summation takes over all vertices of $C_q$. Then $M(L(B(u,w)))=0$ and  $M(L(B(v',w)))$ is odd by Theorem \ref{art28}. It follows that  $M(L(B))$ is odd by the above equation and  $B\in \mathscr{B}_{2}$ is $(1,2)$-choosable by Theorem \ref{art22}.
\end{proof}

Let $C_p$ and $C_q$ denote the induced cycle of any bicyclic graph in $\mathscr{B}_n^2(p,q)$. Let $\mathscr{B}_{3}\subset \mathscr{B}_n^2(p,q,r)$
be the set of all graphs obtained
by identifying every vertex $v_i$ of $B_2(p,q,r)(r>2)$ and
any leaf of each of the $k_i$ trees,
where $k_i=k_i^0+k_i^1$
(the number of trees with even number of edges $(\geq 2)$ is denoted by $k_i^0$,
the number of trees with odd number of edges $(\geq 3)$ is denoted by $k_i^1$)
and $\sum_{i} k_i^0=s_1$,
where the summation takes over all vertices of $C_p$
if $m(B(u,w))$ is even, $s$ is odd or
$C_q$ if $m(B(u,w))$ is odd, $s$ is even.
Because we need to consider the number of perfect matches of line graph,
$n$ is assumed to be odd.

\begin{lemma}\label{art211}
Let $B\in \mathscr{B}_{3}$ be a bicyclic graph. Then $B$ is $(1,2)$-choosable.
\end{lemma}
\begin{proof} Let $e=uz$. By Lemma \ref{art26},
\begin{eqnarray*}
M(L(B))=M(L(B(u,w)))+M(L(B(z,w))).
\end{eqnarray*}

Obviously, $B(u,w), B(z,w)\in \mathscr{U}_{n}$.
According to the definition of $\mathscr{B}_{3}$, $m(B)$ is even.
We consider two cases as follows.

\textbf{Case 1:} $m(B(u,w))$ is even and $m(B(z,w))$ is odd.

By the above definition, we have $\sum_{i} k_i^0=s_1$ is odd,
where the summation takes over all vertices of $C_p$. Then $M(L(B(z,w)))=0$.
By Theorem \ref{art28}, $M(L(B(u,w)))$ is odd. It follows that  $M(L(B))$ is odd by the above equation  and  $B\in \mathscr{B}_{3}$ is $(1,2)$-choosable by Theorem \ref{art22}.

\textbf{Case 2:} $m(B(u,w))$ is odd and $m(B(z,w))$ is even.

By the above definition, we have $\sum_{i} k_i^0=s_1$ is even,
where the summation takes over all vertices of $C_q$. Then $M(L(B(u,w)))=0$.
By Theorem \ref{art28}, $M(L(B(z,w)))$ is odd. It follows that  $M(L(B))$ is odd by the above equation and  $B\in \mathscr{B}_{3}$ is $(1,2)$-choosable by Theorem \ref{art22}.
\end{proof}

Let $\mathscr{B}_{4}\subset \mathscr{B}_n^3(p,q,r)$
be the set of all graphs obtained
by identifying every vertex $v_i$ of $B_3(p,q,r)$ with
any leaf of each of the $k_i$ trees,
where $k_i=k_i^0+k_i^1$
(the number of trees with even number of edges $(\geq 2)$ is denoted by $k_i^0$,
the number of trees with odd number of edges $(\geq 3)$ is denoted by $k_i^1$)
and $\sum_{i} k_i^0=s_1$ is odd or even,
where the summation takes over all vertices of $P_p$, $P_q$ and $u,v$.
Because we shall think about the number of perfect matchings of line graph,
$n$ is assumed to be odd.

\begin{lemma}\label{art212}
Let $B\in \mathscr{B}_{4}$ be a bicyclic graph. Then $B$ is $(1,2)$-choosable.
\end{lemma}

\begin{proof}Let $e=uy$. By Lemma \ref{art26},
\begin{eqnarray*}
M(L(B))=M(L(B(u,w)))+M(L(B(y,w))).
\end{eqnarray*}

Clearly, $B(u,w), B(y,w)\in \mathscr{U}_{n}$.
By the definition of $\mathscr{B}_{4}$, $m(B)$ is even.
Then the number of edges of $B(u,w)$ and $B(y,w)$ are both even.
Set $s=x'$ for $B(u,w)$ and $s=y'$ for $B(y,w)$.

Let $B'$ be a subgraph of $B$
induced by $P_{r}$, edges $uy,vx$ and the trees hanging
on all vertices of $P_{r}$.
We consider two cases as follows.

\textbf{Case 1:} $m(B')$ is even.

If $s_1$ is odd, then $x'=s_1$
and $y'=s_1+1$. Then $M(L(B(u,w)))$ is odd by Theorem \ref{art28} and  $M(L(B(y,w)))$ is even by Lemma \ref{art27}. It follows that $M(L(B))$ is odd by the above equation  and $B\in \mathscr{B}_{4}$ is $(1,2)$-choosable by Theorem \ref{art22}.

If $s_1$ is even, then $x'=s_1$
and $y'=s_1+1$.
Hence, $M(L(B(u,w)))$ is even by Lemma \ref{art27} and  $M(L(B(y,w)))$ is odd by Theorem \ref{art28}. It follows that $M(L(B))$ is odd by the above equation $(7)$ and  $B\in \mathscr{B}_{4}$ is $(1,2)$-choosable by Theorem \ref{art22}.

\textbf{Case 2:} $m(B')$ is odd.

If $s_1$ is odd, then $x'=s_1+1$
and $y'=s_1$.
Hence, $M(L(B(u,w)))$ is even by Lemma \ref{art27}.
According to Theorem \ref{art28}, we have $M(L(B(y,w)))$ is odd. It follows that $M(L(B))$ is odd by the above equation $(7)$ and  $B\in \mathscr{B}_{4}$ is $(1,2)$-choosable by Theorem \ref{art22}.

If $s_1$ is even,then $x'=s_1+1$ and $y'=s_1$.
Hence, $M(L(B(u,w)))$ is  odd by Theorem \ref{art28} and
$M(L(B(y,w)))$ is even by Lemma \ref{art27}.
It follows that $M(L(B))$ is odd by the above equation and  $B\in \mathscr{B}_{4}$ is $(1,2)$-choosable by Theorem \ref{art22}.
\end{proof}

According to Lemmas \ref{art29},\ref{art210},\ref{art211} and \ref{art212},
we obtain the following result in this section.

\begin{theorem}\label{art213}
Let $B\in \mathscr{B}_{1}\cup\mathscr{B}_{2}\cup\mathscr{B}_{3}\cup\mathscr{B}_{4}$
be a bicyclic graph.
Then $B$ is $(1,2)$-choosable.
\end{theorem}

\section{Total weight choosability of $(n,m)$-graphs when $m=n$ and $n+1$}

In this section, we will show that all $(n,m)$-graphs are
$(2,2)$-choosable and $(1,3)$-choosable, where $m=n$ and $n+1$.
Obviously, $(n,m)$-graphs are unicyclic graphs when $m=n$;
$(n,m)$-graphs are bicyclic graphs when $m=n+1$. A \emph{sink} in a digraph is a vertex of outdegree zero. Before the proof of main theorems,
we present some lemmas as follows.

\begin{lemma}(\cite{wong1})\label{art33}
Let an index function $\eta$ be non-singular
if there is a valid index function $\eta'\leq \eta$ with ${\rm per}(A_{G}(\eta'))\neq 0$.
Suppose $G$ is a graph, $\eta$ is an index function of $G$
with $\eta(e)=1$ for all edges $e$,
and $X$ is a subset of $V(G)$.
Let $G'=G-E[X]$ be obtained from $G$
by deleting edges in $G[X]$.
Let $D$ be an acyclic orientation of $G'$,
in which each vertex $v\in X$ is a sink.
Assume that $D'$ is a sub-digraph of $D$
such that for all $v\in V(D)$:
\begin{eqnarray*}
\eta(v)+2d_{D'}^{-}(v)-d_{D}^{-}(v)\geq d_{D'}^{+}(v).
\end{eqnarray*}
Let $\eta'$ be the index function defined as $\eta'(e)=1$
for every edge $e$ of $G[X]$
and $\eta'(v)=\eta(v)+2d_{D'}^{-}(v)-d_{D}^{-}(v)$
for $v\in X$.
If $\eta'$ is a non-singular index function for $G[X]$,
then $\eta$ is a non-singular index function for $G$.
\end{lemma}

\begin{lemma}(\cite{wong2})\label{art34}
Suppose $G$ is obtained from a graph $G'$ by adding one vertex $v$
and one edge $e=uv$, where $u$ is a vertex of $G'$.
If ${\rm pind}(A_{G'})=1$, then ${\rm pind}(A_{G})=1$.
If $G'$ is $(2,2)$-choosable, then $G$ is $(2,2)$-choosable.
\end{lemma}

\begin{lemma}(\cite{bart})\label{art35}
If $T$ is a tree with at least two edges, then ${\rm pind}(B_{T})\leq 2$.
Hence $T$ is $(1,3)$-choosable.
\end{lemma}

A \emph{hanging edge} of a graph $G$ is an edge $e=uv$ of $G$
such that $d_G(v)=1$ and $d_G(u)=2$ or 3.

\begin{lemma}(\cite{wong2})\label{art36}
Let $G$ be a graph containing a hanging edge $e=uv$
and $G'=G-\{u,v\}$.
If ${\rm pind}(B_{G'})\leq 2$, then ${\rm pind}(B_{G})\leq 2$.
\end{lemma}

\begin{lemma}(\cite{zhu})\label{art37}
Let $G'$ be obtained from a graph $G$ by adding two new vertices $u, v$
and two new edges $e_1=uv, e_2=uw$,
where $w\in V(G)$. Then ${\rm pind}(B_{G'})\leq {\rm pind}(B_{G})$.
\end{lemma}

\begin{lemma}(\cite{bart})\label{art38}
Let $G=(V, E)$ be a graph such that ${\rm pind}(B_{G})\leq 2$ and $U$ be a nonempty subset of $V(G)$.
Denote by $F$ the graph obtained by adding two new vertices $u, v$ to $G$
and joining them to each vertex of $U$.
Then ${\rm pind}(B_{F})\leq 2$.
\end{lemma}

A \emph{thread} in a graph $G$ is a path $P=(v_{1},v_{2},\ldots,v_{k})$
in $G$ such that $d_{G}(v_{i})=2$ for $i=2,3,\ldots,k-1$. The vertices $v_{1},v_{k}$ need not to be distinct.
If we need to specify the two end vertices of a thread,
then we say $P$ is a $(v_{1}-v_{k})$-thread.
By \emph{deleting a thread} $P=(v_{1},v_{2},\ldots,v_{k})$ from $G$,
we mean deleting the vertices $v_{2}, v_{3},\ldots, v_{k-1}$
(and hence edges incident to them). The \emph{length} of a thread is the number of edges in it. The notations defined in section~2 are followed in this section.

\begin{lemma}(\cite{wong2})\label{art39}
Let $G'$ be obtained from a graph $G$
by deleting a thread of length 4.
If ${\rm pind}(B_{G'})\leq 2$, then ${\rm pind}(B_{G})\leq 2$.
\end{lemma}

\begin{theorem}\label{art310}
Let $U\in\mathscr{U}_{n}$, then ${\rm pind}(A_{U})=1$.
Hence $U$ is $(2,2)$-choosable.
\end{theorem}

\begin{proof}
According to Lemma \ref{art34}, it is sufficient to prove this Theorem holds for the unique cycle $C_l$ in $U$.

First we construct an acyclic orientation of $C_l$ as follows:
orient the edges $v_iv_{i+1}(i=1,2, \ldots, l-1)$
from $v_i$ to $v_{i+1}$
and orient the edge $v_1v_{l}$ from $v_1$ to $v_{l}$.
The resulting digraph is denoted by $D$ and $v_l$ is a sink vertex in $D$ obviously. Let $D'$ be a sub-digraph of $D$ consisting of the edge $v_1v_{l}$.
Let $\eta\equiv 1$ be a constant function,
$X=\{v_l\}$ and $\eta'(v_l)=0$ be an index function of $U[X]$.
Because there exist a valid index function $\eta''\leq \eta'$ with ${\rm per}(A_{U[X]}(\eta''))\neq 0$.
Then $\eta'$ is a non-singular index function of $U[X]$.
To prove that ${\rm pind}(A_{U})=1$, i.e.,
$\eta$ is a non-singular index function of $U$,
is suffices, by Lemma \ref{art33},
to show that for each vertex $v$,
\begin{eqnarray*}
1+2d_{D'}^{-}(v)-d_{D}^{-}(v)\geq d_{D'}^{+}(v).
\end{eqnarray*}

We show that every vertex $v$ of $U$ satisfies the above equation and consider three cases.

\textbf{Case 1:} $v=v_1$.

Then $d_{D'}^{-}(v)=d_{D}^{-}(v)$=0, $d_{D'}^{+}(v)=1$.
So $1+2d_{D'}^{-}(v)-d_{D}^{-}(v)=1
\geq d_{D'}^{+}(v)$.

\textbf{Case 2:} $v=v_i(i=2,3,\ldots, l-1)$.

Then $d_{D'}^{-}(v)=d_{D'}^{+}(v)=0, d_{D}^{-}(v)=1$.
So $1+2d_{D'}^{-}(v)-d_{D}^{-}(v)=0
\geq d_{D'}^{+}(v)$.

\textbf{Case 3:} $v=v_l$.

Then
$d_{D'}^{-}(v)=1, d_{D}^{-}(v)=2, d_{D'}^{+}(v)=0$.
So $1+2d_{D'}^{-}(v)-d_{D}^{-}(v)=1
\geq d_{D'}^{+}(v)$.
\end{proof}

\begin{theorem}\label{art311}
Let $U\in\mathscr{U}_{n}$, then ${\rm pind}(B_{U})\leq 2$.
Hence $U$ is $(1,3)$-choosable.
\end{theorem}

\begin{proof}
We consider three cases as follows.

\textbf{Case 1:} $U=C_l$.

Then ${\rm pind}(B_{U})\leq 2$ by Wong and Zhu in \cite{wong2}.
Hence $U$ is $(1,3)$-choosable.

\textbf{Case 2:} $U$ is a graph obtained
by identifying vertex $v_i$ of $C_l$ with
the center of a star $K_{1,s_i}$, where $1\leq i\leq l$.

First consider the case $i=1$. If $s_i=1$, then ${\rm pind}(B_{U})\leq 2$ according to the direct calculation.
Hence $U$ is $(1,3)$-choosable.
Based on the above result and Lemma \ref{art38}, ${\rm pind}(B_{U})\leq 2$ and hence $U$ is $(1,3)$-choosable if $s_i\geq 2$.  The proof of the case $i=2$
is similar to $i=1$ and is thus omitted.
By repeating the above process,
we can prove the above Theorem holds for $i=3, 4,\ldots,l$.

\textbf{Case 3:} $U$ is a graph obtained
by identifying vertex $v_i$ of $C_l$ with
a vertex of a tree $T_{s_i}$, where $1\leq i\leq l$
and at least one $T_{s_i}$ is not a star.

We prove the Theorem by induction on $m'=|\bigcup\limits_{i=1}^{l}E(T_{s_i})|$.
If $m'=1$, then ${\rm pind}(B_{U})\leq 2$ according to Case $2$.
Assume that the above Theorem holds for the number of edges in  the hanging trees
less than $m'(\geq 2)$. Consider the case that $|\bigcup\limits_{i=1}^{l}E(T_{s_i})|=m'$.
By induction hypothesis, the above Theorem holds for $m'-2$ and ${\rm pind}(B_{U})\leq 2$
by Lemmas \ref{art37} and \ref{art38}. Hence $U$ is $(1,3)$-choosable.
\end{proof}

\begin{theorem}\label{art312}
Let $B\in\mathscr{B}_{n}$, then ${\rm pind}(A_{B})=1$.
Hence $B$ is $(2,2)$-choosable.
\end{theorem}
\begin{proof}
According to Lemma \ref{art34},
we only need to prove this Theorem holds for $B_1(p,q)$, $B_2(p,q,\\r)$, $B_3(p,q,r)$.
We consider three cases.

First we construct an acyclic orientation of $B_1(p,q)$ as follows:
For $C_p$, except for the clockwise orientation of edge $v_1u$,
all the other edges are oriented anticlockwise;
For $C_q$, except for the anticlockwise orientation of edge $v_2u$,
all the other edges are oriented clockwise.
The resulting digraph is denoted by $D$. Let $D'$ be a sub-digraph of $D$ consisting of the edges $v_{1}u, v_2u$.
It is easy to see that $u$ is a sink of $D$.
Let $\eta\equiv 1$ be a constant function, $X=\{u\}$ and $\eta'(u)=0$ be an index function of $B[X]$.
Because there exist a valid index function $\eta''\leq \eta'$ with ${\rm per}(A_{B[X]}(\eta''))\neq 0$.
Then $\eta'$ is a non-singular index function of $B[X]$.
To prove that ${\rm pind}(A_{B})=1$, i.e.,
$\eta$ is a non-singular index function of $B$,
is suffices, by Lemma \ref{art33},
to show that for each vertex $v$,
\begin{eqnarray*}
1+2d_{D'}^{-}(v)-d_{D}^{-}(v)\geq d_{D'}^{+}(v).
\end{eqnarray*}

We show that every vertex $v$ of $B$ satisfies the above equation by considering three cases.

\textbf{Case 1:} $v=u$.

Then $d_{D'}^{-}(v)=2, d_{D}^{-}(v)$=4, $d_{D'}^{+}(v)=0$.
So $1+2d_{D'}^{-}(v)-d_{D}^{-}(v)=1
\geq d_{D'}^{+}(v)$.

\textbf{Case 2:} $v= v_1, v_2$.

Then $d_{D'}^{-}(v)=0, d_{D}^{-}(v)=0, d_{D'}^{+}(v)=1$.
So $1+2d_{D'}^{-}(v)-d_{D}^{-}(v)=1
\geq d_{D'}^{+}(v)$.

\textbf{Case 3:} $v\in B\setminus \{u,v_1,v_2\}$.

Then $d_{D'}^{-}(v)=0, d_{D}^{-}(v)=1, d_{D'}^{+}(v)=0$.
So $1+2d_{D'}^{-}(v)-d_{D}^{-}(v)=0
\geq d_{D'}^{+}(v)$.

Secondly, we construct an acyclic orientation of $B_2(p,q,r)$ as follows:
For $C_p$, except for the clockwise orientation of edge $v_3u$,
all the other edges are oriented anticlockwise;
For $C_q$, except for the anticlockwise orientation of edge $v_4v'$,
all the other edges are oriented clockwise;
For $P_r$, orient all edges from right to left.
The resulting digraph is denoted by $D$. Let $D'$ be a sub-digraph of $D$ consisting of edges $v_3u, v_4v'$.
It is easy to see that there $u$ is a sink vertex of $D$.
Let $\eta\equiv 1$ be a constant function, $X=\{u\}$ and $\eta'(u)=0$ be an index function of $B[X]$.
Because there exist a valid index function $\eta''\leq \eta'$ with ${\rm per}(A_{B[X]}(\eta''))\neq 0$.
Then $\eta'$ is a non-singular index function of $B[X]$.
To prove that ${\rm pind}(A_{B})=1$, i.e.,
$\eta$ is a non-singular index function of $B$,
is suffices, by Lemma \ref{art33},
to show that for each vertex $v$,
\begin{eqnarray*}
1+2d_{D'}^{-}(v)-d_{D}^{-}(v)\geq d_{D'}^{+}(v).
\end{eqnarray*}

We show that every vertex of $B$
satisfies the above equation by considering four cases.

\textbf{Case 1:} $v=u$.

Then $d_{D'}^{-}(v)=1, d_{D}^{-}(v)$=3, $d_{D'}^{+}(v)=0$.
So $1+2d_{D'}^{-}(v)-d_{D}^{-}(v)=0
\geq d_{D'}^{+}(v)$.

\textbf{Case 2:} $v=v'$.

Then $d_{D'}^{-}(v)=1, d_{D}^{-}(v)=2, d_{D'}^{+}(v)=0$.
So $1+2d_{D'}^{-}(v)-d_{D}^{-}(v)=1
\geq d_{D'}^{+}(v)$.

\textbf{Case 3:} $v= v_3, v_4$.

Then
$d_{D'}^{-}(v)=0, d_{D}^{-}(v)=0, d_{D'}^{+}(v)=1$.
So $1+2d_{D'}^{-}(v)-d_{D}^{-}(v)=1
\geq d_{D'}^{+}(v)$.

\textbf{Case 4:} $v\in B\setminus\{u,v',v_3,v_4\}$.

Then $d_{D'}^{-}(v)=0, d_{D}^{-}(v)=1, d_{D'}^{+}(v)=0$.
So $1+2d_{D'}^{-}(v)-d_{D}^{-}(v)=0
\geq d_{D'}^{+}(v)$.

Finally, we construct an acyclic orientation of $B_3(p,q,r)$ as follows:
For the cycle consisting of $P_p$, $P_q$ and
the edges $v_5u$, $v_8u$, $v_6v'$, $v_7v'$,
except for the anticlockwise orientation of edges $v_8u$, $v_6v'$,
all the other edges are oriented clockwise;
For $P_r$ and the edges $uy$, $xv'$,
orient all edges from top to bottom.
The resulting digraph is denoted by $D$. Let $D'$ be a sub-digraph of $D$ consisting of the edges $v_8u, v_6v', uy$.
It is easy to see that $u$ is a sink vertex of $D$.
Let $\eta\equiv 1$ be a constant function, $X=\{u\}$ and $\eta'(u)=0$ be an index function of $B[X]$.
Because there exist a valid index function $\eta''\leq \eta'$ with ${\rm per}(A_{B[X]}(\eta''))\neq 0$.
Then $\eta'$ is a non-singular index function of $B[X]$.
To prove that ${\rm pind}(A_{B})=1$, i.e.,
$\eta$ is a non-singular index function of $B$,
is suffices, by Lemma \ref{art33},
to show that for each vertex $v$,
\begin{eqnarray*}
1+2d_{D'}^{-}(v)-d_{D}^{-}(v)\geq d_{D'}^{+}(v).
\end{eqnarray*}

We show that every vertex $v$ of $B$
satisfies the above equation by considering five cases.

\textbf{Case 1:} $v=v_5,v_8$.

Then $d_{D'}^{-}(v)=d_{D}^{-}(v)$=0, $d_{D'}^{+}(v)=1$.
So $1+2d_{D'}^{-}(v)-d_{D}^{-}(v)=1\geq d_{D'}^{+}(v)$.

\textbf{Case 2:} $v=v'$.

Then
$d_{D'}^{-}(v)=1, d_{D}^{-}(v)=3, d_{D'}^{+}(v)=0$.
So $1+2d_{D'}^{-}(v)-d_{D}^{-}(v)=0\geq d_{D'}^{+}(v)$.

\textbf{Case 3:} $v=y$.

Then $d_{D'}^{-}(v)=d_{D}^{-}(v)=1, d_{D'}^{+}(v)=0$.
So $1+2d_{D'}^{-}(v)-d_{D}^{-}(v)=2\geq d_{D'}^{+}(v)$.

\textbf{Case 4:} $v=u$.

Then $d_{D'}^{-}(v)=1, d_{D}^{-}(v)=2, d_{D'}^{+}(v)=1$.
So $1+2d_{D'}^{-}(v)-d_{D}^{-}(v)=1\geq d_{D'}^{+}(v)$.

\textbf{Case 5:} $v\in B\setminus \{v_5,v_8,v',y,u\}$.

Then $d_{D'}^{-}(v)=0, d_{D}^{-}(v)=1, d_{D'}^{+}(v)=0$.
So $1+2d_{D'}^{-}(v)-d_{D}^{-}(v)=0\geq d_{D'}^{+}(v)$.
\end{proof}

\begin{theorem}\label{art313}
Let $B\in\mathscr{B}_{n}$, then ${\rm pind}(B_{B})\leq 2$.
Hence $B$ is $(1,3)$-choosable.
\end{theorem}

\begin{proof}
Due to $\mathscr{B}_{n}=\mathscr{B}_n^1(p,q)
\cup\mathscr{B}_n^2(p,q,r)\cup\mathscr{B}_n^3(p,q,r)$,
we consider three cases:

\textbf{Case 1:} $B\in\mathscr{B}_n^1(p,q)$.

If $B\in\mathscr{B}_n^1(p,q)-\mathscr{B}_n^1(3,3)$, then
we consider three subcases.

\textbf{Subcase 1.1:} $B=B_1(p,q)$.

According to Lemma \ref{art39} and Theorem \ref{art311},
we can obtain that ${\rm pind}(B_{B})\leq 2$.
Hence $B$ is $(1,3)$-choosable.

\textbf{Subcase 1.2:} $B$ is a graph obtained
by identifying vertex $v_i$ of $B_1(p,q)$ and
the center of a star $K_{1,s_i}$.

If there is no star hanging on the vertex $u$.
First  consider the case $i=1$.
According to Lemma \ref{art36} and Theorem \ref{art311},  ${\rm pind}(B_{B})\leq 2$ if $s_i=1$ and hence $B$ is $(1,3)$-choosable.
Based on the above result and Lemma \ref{art38}, ${\rm pind}(B_{U})\leq 2$ and
hence $B$ is $(1,3)$-choosable if $s_i\geq 2$.
The proof of the case $i=2$
is similar to $i=1$ and is thus omitted.
By repeating the above process,
we can prove the above Theorem holds for all $i\geq 3$.
So, ${\rm pind}(B_{B})\leq 2$,
hence $B$ is $(1,3)$-choosable.

Assume that there exists a star hanging on the vertex $u$.
First we consider the case that hanging stars exist only on vertex $u$.
According to Lemma \ref{art39} and Theorem \ref{art311},
we obtain that if $s_i=1$, then ${\rm pind}(B_{B})\leq 2$ and
hence $B$ is $(1,3)$-choosable.
Based on the above result and Lemma \ref{art38}, ${\rm pind}(B_{B})\leq 2$ and hence $B$ is $(1,3)$-choosable if $s_i\geq 2$.
The proof of the cases that hanging trees exist on other vertices
is quite similar to that on $u$ and is thus omitted.
By repeating the above process,
we can prove the above Theorem holds for other vertices.
So, ${\rm pind}(B_{B})\leq 2$,
hence $B$ is $(1,3)$-choosable.

\textbf{Subcase 1.3:} $B$ is a graph obtained
by identifying vertex $v_i$ of $B_1(p,q)$ and
any vertex of a tree $T_{s_i}$, where $1\leq i\leq p+q-1$
and at least one $T_{s_i}$ is not a star.

We prove the Theorem by induction on $m'$,
which is the number of edges of the hanging trees.
If $m'=1$, ${\rm pind}(B_{B})\leq 2$ according to Case $1.2$.
Assume that the theorem holds if the number of edges in the hanging trees less than $m'$.
If the number of edges of the hanging trees is $m'$,
then the above Theorem holds for $m'-2$ by induction hypothesis and ${\rm pind}(B_{B})\leq 2$ by Lemmas \ref{art37} and \ref{art38}. Hence $B$ is $(1,3)$-choosable.

Assume that $B\in\mathscr{B}_n^1(3,3)$. By direct calculation, ${\rm pind}(B_{B_1(3,3)})\leq 2$. The proof of ${\rm pind}(B_{B})\leq 2$ is quite similar to the case $B\in\mathscr{B}_n^1(p,q)-\mathscr{B}_n^1(3,3)$
and is thus omitted. Hence $B$ is $(1,3)$-choosable.

\textbf{Case 2:} $B\in\mathscr{B}_n^2(p,q,r)$.

If $B\in\mathscr{B}_n^2(p,q,r)-\mathscr{B}_n^2(3,3,2)
-\mathscr{B}_n^2(3,3,3)-\mathscr{B}_n^2(3,3,4)$,
then consider three subcases.

\textbf{Subcase 2.1:} $B=B_2(p,q,r)$.

According to Lemma \ref{art39} and Theorem \ref{art311},
we can obtain that ${\rm pind}(B_{B})\leq 2$ and
hence $B$ is $(1,3)$-choosable.

\textbf{Subcase 2.2:} $B$ is a graph obtained
by identifying  vertex $v_i$ of $B_2(p,q,r)$ with
the center of a star $K_{1,s_i}$.

If there is no star hanging on vertices $u,v'$.
First we consider the case $i=1$.
According to Lemma \ref{art36} and Theorem \ref{art311},
we obtain that ${\rm pind}(B_{B})\leq 2$ if $s_i=1$ and
hence $B$ is $(1,3)$-choosable.
Based on the above result and Lemma \ref{art38},
 then ${\rm pind}(B_{U})\leq 2$ and
hence $B$ is $(1,3)$-choosable if $s_i\geq 2$.
The proof of case $i=2$
is quite similar to $i=1$ and is thus omitted.
By repeating the above process,
we can prove the above Theorem holds for $i\geq 3$.
So, ${\rm pind}(B_{B})\leq 2$ and
hence $B$ is $(1,3)$-choosable.

Without loss of generality, assume there has a star hanging on vertex $u$
and there is no star hanging on vertices $v'$ .
First we consider the case only hanging star on vertex $u$.
According to Lemma \ref{art39} and Theorem \ref{art311},
we obtain that if $s_i=1$, then ${\rm pind}(B_{B})\leq 2$,
hence $B$ is $(1,3)$-choosable.
Based on the above result and Lemma \ref{art38},
we have if $s_i$ is odd, then ${\rm pind}(B_{B})\leq 2$,
hence $B$ is $(1,3)$-choosable;
by Lemma \ref{art38}, we obtain that if $s_i$ is even,
then ${\rm pind}(B_{B})\leq 2$,hence $B$ is $(1,3)$-choosable.
Hence, if $i=1$, then ${\rm pind}(B_{B})\leq 2$,
so $B$ is $(1,3)$-choosable.
The proof of the cases of other vertices
is quite similar to $u$ and is thus omitted.
By repeating the above process,
we can prove the above Theorem holds for other vertices.
So, ${\rm pind}(B_{B})\leq 2$,
hence $B$ is $(1,3)$-choosable.

Assume that there exists a star hanging on vertices $u$ and $v'$.
First we consider the case that there is a hanging star on just one of the vertices $u$ and $v'$. Without loss of generality, we assume the vertex to be $u$.
According to Lemma \ref{art39} and Theorem \ref{art311}, ${\rm pind}(B_{B})\leq 2$ if $s_i=1$ and hence $B$ is $(1,3)$-choosable.
Based on the above result and Lemma \ref{art38}, ${\rm pind}(B_{B})\leq 2$ and
hence $B$ is $(1,3)$-choosable if $s_i\geq 2$. Now consider the case that there are hanging stars on both of the vertices $u$ and $v'$. Similar to the proof in the case of $u$,
we can prove the theorem holds for $v'$ and other vertices and the details are omitted.
So, ${\rm pind}(B_{B})\leq 2$,
hence $B$ is $(1,3)$-choosable.

\textbf{Subcase 2.3:} $B$ is a graph obtained
by identifying  vertex $v_i$ of $B_2(p,q,r)$ with
the vertex of $T_{s_i}$, where $1\leq i\leq p+q+r-2$
and at least one $T_{s_i}$ is not a star.

We prove the Theorem by induction on $m'$,
which is the number of edges of the hanging trees.
If $m'=1$, ${\rm pind}(B_{B})\leq 2$ according to Case $2.2$.
Assume that the above Theorem holds if the number of edges in the hanging trees
less than $m'$. Consider the case that the number of edges of the hanging trees is $m'$.
By induction hypothesis, the theorem holds for $m'-2$.
By Lemmas \ref{art37} and \ref{art38}, ${\rm pind}(B_{B})\leq 2$ and
hence $B$ is $(1,3)$-choosable.

Assume that $B\in\mathscr{B}_n^2(3,3,2)
\cup\mathscr{B}_n^2(3,3,3)\cup\mathscr{B}_n^2(3,3,4)$.
By direct calculating,
we can obtain that ${\rm pind}(B_{B_2(3,3,2)})\leq 2$, ${\rm pind}(B_{B_2(3,3,3)})\leq 2$, ${\rm pind}(B_{B_2(3,3,4)})\leq 2$. The proof of ${\rm pind}(B_{B})\leq 2$ is similar to the case that $B\in\mathscr{B}_n^2(p,q,r)-\mathscr{B}_n^2(3,3,2)
-\mathscr{B}_n^2(3,3,3)-\mathscr{B}_n^2(3,3,4)$
and is thus omitted. So, ${\rm pind}(B_{B})\leq 2$,
hence $B$ is $(1,3)$-choosable.

\textbf{Case 3:} $B\in\mathscr{B}_n^3(p,q,r)$.

If $B\in\mathscr{B}_n^3(p,q,r)-B_n^3(1,1,1)-B_n^3(1,1,2)$,
then consider three cases.

\textbf{Subcase 3.1:} $B=B_3(p,q,r)$.

According to Lemma \ref{art39} and Theorem \ref{art311},
we can obtain that ${\rm pind}(B_{B})\leq 2$ and
hence $B$ is $(1,3)$-choosable.

\textbf{Subcase 3.2:} $B$ is a graph obtained
by identifying  vertex $v_i$ of $B_3(p,q,r)$ and
the center of a star $K_{1,s_i}$.

If there is no star hanging on vertices $u,v'$.
First consider the case $i=1$.
According to Lemma \ref{art36} and Theorem \ref{art311}, ${\rm pind}(B_{B})\leq 2$ if $s_i=1$ and hence $B$ is $(1,3)$-choosable.
Based on the above result and Lemma \ref{art38},
${\rm pind}(B_{B})\leq 2$ and
hence $B$ is $(1,3)$-choosable if $s_i\geq 2$.
The proof of the case $i=2$
is similar to $i=1$ and is thus omitted.
By repeating the above process,
we can prove the theorem holds for $i\geq 3$.
So, ${\rm pind}(B_{B})\leq 2$ and
hence $B$ is $(1,3)$-choosable.

Without loss of generality, assume there has a star hanging on vertex $u$
and there is no star hanging on vertices $v'$ .
First we consider the case only hanging star on vertex $u$.
According to Lemma \ref{art39} and Theorem \ref{art311},
we obtain that if $s_i=1$, then ${\rm pind}(B_{B})\leq 2$,
hence $B$ is $(1,3)$-choosable.
Based on the above result and Lemma \ref{art38},
we have if $s_i$ is odd, then ${\rm pind}(B_{B})\leq 2$,
hence $B$ is $(1,3)$-choosable;
by Lemma \ref{art38}, we obtain that if $s_i$ is even,
then ${\rm pind}(B_{B})\leq 2$,hence $B$ is $(1,3)$-choosable.
Hence, if $i=1$, then ${\rm pind}(B_{B})\leq 2$,
so $B$ is $(1,3)$-choosable.
Based on the above Theorem holds in the case of $u$,
we can prove the above Theorem holds for other vertices.
The proof of the cases of other vertices
is quite similar to $u$ and is thus omitted.
By repeating the above process,
we can prove the above Theorem holds for other vertices.
So, ${\rm pind}(B_{B})\leq 2$,
hence $B$ is $(1,3)$-choosable.

Assume that there exists a star hanging on vertices $u$ and $v'$.
First we consider the case that there is a hanging star on just one of the vertices $u$ and $v'$. Without loss of generality, we assume the vertex to be $u$.
According to Lemma \ref{art39} and Theorem \ref{art311}, ${\rm pind}(B_{B})\leq 2$ if $s_i=1$ and hence $B$ is $(1,3)$-choosable.
Based on the above result and Lemma \ref{art38}, ${\rm pind}(B_{B})\leq 2$ and
hence $B$ is $(1,3)$-choosable if $s_i\geq 2$. Now consider the case that there are hanging stars on both of the vertices $u$ and $v$. Similar to the proof in the case of $u$,
we can prove the theorem holds for $v$ and other vertices and the details are omitted.
So, ${\rm pind}(B_{B})\leq 2$,
hence $B$ is $(1,3)$-choosable.

\textbf{Subcase 3.3:} $B$ is a graph obtained
by identifying vertex $v_i$ of $B_3(p,q,r)$ and
the vertex of a star $T_{s_i}$, where $1\leq i\leq p+q+r+2$
and at least one $T_{s_i}$ is not a star.

We prove the theorem by induction on $m'$,
which is the number of edges of the hanging trees.
If $m'=1$, then
${\rm pind}(B_{B})\leq 2$ according to Case $3.2$.
Assume that the theorem holds if the number of edges in the hanging trees less than $m'$.
If the number of edges of the hanging trees is $m'$, the theorem holds for $m'-2$ by induction hypothesis. By Lemmas \ref{art37} and \ref{art38}, ${\rm pind}(B_{B})\leq 2$ and hence $B$ is $(1,3)$-choosable.

Assume that $B\in\mathscr{B}_n^3(1,1,1)
\cup\mathscr{B}_n^3(1,1,2)$.
By direct calculating, ${\rm pind}(B_{B_3(1,1,1)})\leq 2$, ${\rm pind}(B_{B_3(1,1,2)})\leq 2$.
The proof of ${\rm pind}(B_{B})\leq 2$ is similar to the case
$B\in\mathscr{B}_n^3(p,q,r)-B_n^3(1,1,1)-B_n^3(1,1,2)$
and is thus omitted. Hence $B$ is $(1,3)$-choosable.

The proof of the theorem is complete.
\end{proof}

\begin{remark}
For graph $B_1(3,3)$, we are clockwise oriented for the two $C_3$ in it,
then we have matrix $B_{B_1(3,3)}$.
For $B_{B_1(3,3)}$, we select the first column twice, the second column twice and the firth column twice,
then we form a new matrix $B$. The matrices $B_{B_1(3,3)}$ and $B$ are depicted as follows:
$$
 \begin{gathered}
 \setlength{\arraycolsep}{3pt}
B_{B_1(3,3)}=
\begin{bmatrix}
 0& 1 &-1&0& 0 &0\\
 -1& 0 &1&1& 1 &0\\
 1& -1 &0&-1& -1 &0\\
 0& -1 &-1&0& -1&1\\
 0& 1 &1&1& 0 &-1\\
 0& 0 &0&-1& 1 &0
 \end{bmatrix},
 \end{gathered}
 \begin{gathered}
 \setlength{\arraycolsep}{3pt}
B=
\begin{bmatrix}
 0&0 &1&1& 0 &0\\
 -1&-1 &0&0& 1 &1\\
 1& 1 &-1&-1& -1 &-1\\
 0& 0 &-1&-1& 0&0\\
 0& 0 &1&1& 1 &1\\
 0& 0 &0&0& -1 &-1
 \end{bmatrix}
 \end{gathered}
$$
By direct calculation, we have ${\rm per}(B_{B_1(3,3)})=-8\neq 0$.
According to the definition of permanent index of $B_G$,
then ${\rm pind}(B_{B_1(3,3)})\leq 2$.
\end{remark}

\section{Total weight choosability of some graphs under
some graph decorations}
In this section, we prove that
some graphs obtained by some graph operations are $(2,2)$-choosable.
At first, we give the definitions of some graph operations of a connected graph $G$
as follows.

\begin{itemize}
 \item $L(G)$: The vertices of $L(G)$ are the edges of $G$.
  Two edges of $G$ that share a vertex are considered to be
  adjacent in $L(G)$.
  \item $R(G)$: $R(G)$ is obtained from $G$ by adding $|E(G)|$ new vertices and joining each of them to the endvertices of exactly one edge in $E(G)$.
  \item $Q(G)$: $Q(G)$ is obtained from $G$ by inserting a new vertex into each
  edge of $G$, then joining those pairs of new vertices on adjacent edges of $G$ with edges.
\end{itemize}

In order to obtain the main theorems, we present a lemma and a theorem as follows.

\begin{lemma}(\cite{zhu})\label{art41}
Assume that $A$ is an $n\times m$ matrix
and $L$ is an $n\times n$ matrix
whose columns are linear combinations of the columns of $A$.
Let the $j$th column of $A$ be present in $n_j$
such linear combinations (with non-zero coefficients).
Then there is an index function $\eta$:
$\{1,2,\ldots,m\}\rightarrow \{0,1,\ldots\}$
such that $\eta(j)\leq n_j$ and ${\rm per}(A{(\eta)})\neq 0$.
\end{lemma}

\begin{theorem}\label{art42}
Let $G$ be a connected graph,
and $G'$ be a graph obtained
by identifying a vertex of  $G$ with a vertex of  $K_3$.
If ${\rm pind}(A_{G})=1$, then ${\rm pind}(A_{G'})=1$
and hence $G'$ is $(2,2)$-choosable.
\end{theorem}

\begin{proof}
Assume that $G$ is a graph with $n$ vertices, $m$ edges
and there exists an orientation of $G$ such that ${\rm pind}(A_{G})=1$.
By the definition of $A_{G}$,
$A_{G}$ is an $m\times(m+n)$ matrix.
Therefore, according to the definition of permanent index of $A_{G}$
and the assumption,
$A_{G}$ has an $m\times m$ submatrix $B'$
such that ${\rm per}(B')\neq 0$
and each column of $B'$ is a column of $A_{G}$,
each column of $A_{G}$ occurs in $B'$ at most once.

For $K_3$, the vertices are $v_1,v_2,v_3$,
and the edges are $e_{ij}=v_iv_j$ for $1\leq j< i\leq 3$.
According to the assumption of $G'$,
the new added edges of $G'$ are the edges $e_{32}, e_{31}, e_{21}$.
Firstly, we construct an orientation of the new added edges of $G'$.
For $j<i$, we orient the edge $e_{ij}$ from $v_i$ to $v_j$.

According to the assumption of $G'$ and the definition of $A_{G'}$,
$A_{G}$ is an $(m+3)\times(m+n+3)$ matrix.
Next, we construct an $(m+3)\times(m+3)$ submatrix $B''$ of $A_{G'}$.
Let $B''$ be obtained from $B'$ by adding the three rows
$e_{21}, e_{32}, e_{31}$. and the following three columns:
$A_{v_2},A_{v_3},A_{e_{31}-e_{21}}$.
The matrix $B''$ is depicted as follows.
$$
 \begin{gathered}
 \setlength{\arraycolsep}{5pt}
 B''=\begin{matrix}
 \left[
 \begin{array}{c|c}
 B'& 0\\
 \hline
 A& C\\
 \end{array}
 \right],
 \end{matrix}
 \end{gathered}
$$
where
$$
 \begin{gathered}
 \setlength{\arraycolsep}{5pt}
C=
\begin{bmatrix}
 1& -1 &-2\\
 0& -1 &-1\\
 -1& 0 &1\\
 \end{bmatrix}.
 \end{gathered}
$$
Obviously, ${\rm per}(B'')={\rm per}(B'){\rm per}(C)$.
By direct calculation, we have ${\rm per}(C)\neq 0$.
Therefore, according to the assumption of ${\rm per}(B')\neq 0$,
we have ${\rm per}(B'')\neq 0$.
Since each column of $B''$ is linear combination of columns of $A_{G'}$
and each column of $A_{G}$ occurs once in such linear combinations.
By Lemma \ref{art41} and the definition of permanent of $A_{G'}$,
we obtain that ${\rm pind}(A_{G'})=1$.
\end{proof}

\begin{theorem}\label{art43}
Let $G$ be a connected graph
and $G'$ be a graph obtained
by identifying a vertex of $G$ and a vertex of $K_n$.
If $G$ is $(2,2)$-choosable, then $G'$ is $(2,2)$-choosable.
\end{theorem}

\begin{proof}
We prove this theorem by induction on $n$.
$n=1$ is trivial. By Lemma \ref{art34},
this Theorem holds for $n=2$.
Assume that the theorem holds for $n=k-1$.

Now, we consider the case $n=k$.
Let $G''$ be a graph obtained by identifying a vertex of
$G$ and a vertex of $K_{k-1}$, where $V(K_{k-1})=\{v_1, v_2,\ldots v_{k-1}\}$.
By induction, we can obtain that $G''$ is $(2,2)$-choosable.
Hence, there exists a $(2,2)$-total-weight-list assignment of $G''$ denoted by $L$
such that there exists a proper $L$-total-weighting of $G''$.
Next, based on the proper $L$-total weighting of $G''$, we need to find a proper $L'$-total weighting of $G'$.
Checking the structure of $G'$,
$G'$ be a graph obtained from $G''$ by adding a new vertex $v_{k}$ and some new edges
which joining $v_{k}$ to all vertices in $V(K_{k-1})$.
For convenience, denote by $e_{1}=v_1v_k,e_{2}=v_2v_k,\ldots,e_{k-1}=v_{k-1}v_k$ the new edges of $G'$.
Let $L'$ be a $(2,2)$-total-list assignment of $G'$, defined as follows:
For $L'(e_{1}), L'(e_{2}),\ldots L'(e_{k-1})$, choose $w_{1}\in L'(e_{1}),w_{2}\in L'(e_{2}),\ldots,w_{k-1}\in L'(e_{k-1})$;
$L'(z)=L(z)$ if $z \notin \{v_1, v_2,\ldots v_{k}, e_1, e_2,\ldots e_{k-1}\}$;
$L'(v_1)=L(v_1)-w_{1},L'(v_2)=L(v_2)-w_{2},\ldots,L'(v_{k-1})=L(v_{k-1})-w_{k-1}$;
For $L'(v_{k})$, $s(v_k)\neq s(v_i), i=1,2,\cdot\cdot\cdot, k-1$.
Checking all adjacent vertices $v, v'$ in $G'$, we have $s(v)\neq s(v')$.
Hence, we find a proper $L'$-total weighting of $G'$.

Based on the above argument, we obtain that $G'$ is $(2,2)$-choosable.
\end{proof}

According to Theorem \ref{art42},
we can obtain the following corollary.

\begin{corollary}\label{art44}
If $T$ is a tree, then ${\rm pind}(A_{R(T)})=1$
and hence $R(T)$ is $(2,2)$-choosable.
\end{corollary}

By Lemma \ref{art34} and Theorem \ref{art43},
we can get the following corollary naturally.

\begin{corollary}\label{art45}
If $T$ is a tree, then $L(T)$ and $Q(T)$ are both $(2,2)$-choosable.
\end{corollary}

\noindent{\bf Conflict of Interest Statement}

The authors declare that they have no conflicts of interest.

\end{document}